\documentclass{amsart}

\usepackage{amssymb}

 \usepackage{graphicx}

\usepackage[cmtip,all]{xy}

\theoremstyle{plain}
\newtheorem{theorem}{Theorem}[section]
\newtheorem{corollary}[theorem]{Corollary}
\newtheorem{lemma}[theorem]{Lemma}
\newtheorem{proposition}[theorem]{Proposition}

\theoremstyle{definition}

\newtheorem{remark}[theorem]{Remark}

\renewcommand{\leq}{\leqslant}
\renewcommand{\geq}{\geqslant}\usepackage{amssymb}

\newcommand{\vr}{\varepsilon}
\newcommand{\one}{\mathbf{1}}

\newcommand{\be}{\begin{equation}}
\newcommand{\ee}{\end{equation}}

\newcommand{\R}{\mathbb R}

\newcommand{\N}{\mathbb N}

\newcommand{\ball}{\mathbf{B}}
\newcommand{\A}{\mathcal{A}}

\newcommand{\cs}{\mathfrak{C}}
\newcommand{\dist}{\mathrm{dist}}

\def \dim{{\mathrm{dim}} \, }

\def \ker{{\mathrm{ker}} \, }
\def \tr{{\mathrm{Tr}} \, }
\def \Re{{\mathrm{Re}} \, }
\renewcommand{\span}{\mathrm{span}}
\def \vr{\varepsilon}

\def\iK{\mathcal {K}}

\def\iFSS{\mathcal {FSS}}
\def\iSS{\mathcal {SS}}
\def \iIN{\mathcal{IN}}
\def \iWK{\mathcal{WK}}
\def \iDP{\mathcal{DP}}
\def \iSCS{\mathcal{SCS}}

\DeclareMathOperator{\re}{Re}

\renewcommand{\Re}{\re}

\numberwithin{equation}{section}

\begin{document}

% \title[short text for running head]{full title}
\title[Ideals of operators on $C^*$-algebras]{Ideals of operators on $C^*$-algebras
and their preduals}

%    Only \author and \address are required; other information is
%    optional.  Remove any unused author tags.

%    author one information
% \author[short version for running head]{name for top of paper}

\author[T. Oikhberg]{Timur Oikhberg}
\address{
Dept.~of Mathematics, University of Illinois at Urbana-Champaign, Urbana IL 61801, USA}
\email{oikhberg@illinois.edu}
\thanks{}

%    author two information
\author[E. Spinu]{Eugeniu Spinu}
\address{Dept. of Mathematical and Statistical Sciences, University of Alberta
Edmonton, Alberta  T6G 2G1, CANADA}
\email{spinu@ualberta.ca}
\thanks{}
\thanks{The authors acknowledge the generous support of Simons Foundation,
via its Travel Grant 210060. They would also like to thank the organizers of
Workshop in Linear Analysis at Texas A\&M, where part of this work was
carried out.}

%    \subjclass is required.

\date{}

%    "Communicated by" -- provide editor's name; required.
\commby{}

%    Abstract is required.
\begin{abstract}
%Suppose $X$ and $Y$ are either $C^*$-algebras, or preduals of von Neumann algebras,
%and ${\mathcal{I}}$, ${\mathcal{J}}$
%are operator ideals (for instance, the ideals of strictly singular, weakly compact, or
%compact operators). Under what conditions does the inclusion
%${\mathcal{I}}(X,Y) \subset {\mathcal{J}}(X,Y)$, or the equality
%${\mathcal{I}}(X,Y) = {\mathcal{J}}(X,Y)$, hold?
%We also address the same questions for positive parts
%of such ideals.
 In this paper we consider we study various classical operator ideals (for instance, the ideals of strictly (co)singular, weakly compact, Dunford-Pettis operators) either on $C^*$-algebras, or preduals of von Neumann algebras.
\end{abstract}

\maketitle

\section{Introduction and main results}\label{s:intro}

% In this paper, we investigate the following question: suppose ${\mathcal{I}}$ and
% ${\mathcal{J}}$ are two operator ideals, and $X$ and $Y$ are Banach spaces.
% Under what conditions do we have ${\mathcal{I}}(X,Y) = {\mathcal{J}}(X,Y)$?
% If $X$ and $Y$ have positive cones, do we have
% ${\mathcal{I}}(X,Y)_+ = {\mathcal{J}}(X,Y)_+$?
Let $X$ and $Y$ be Banach spaces. We denote by $B(X,Y)$ the set of all linear bounded operators acting between $X$ and $Y$.
In this paper, we address the following questions. When do certain classical operator ideals in $B(X,Y)$ coincide? What are the necessary and sufficient conditions for an operator in $B(X,Y)$ to belong to a given ideal? The spaces $X$ and $Y$ are, mostly,
either  $C^*$-algebras or their (pre)dual.  In particular, they are ordered Banach spaces. So we can consider the same questions for the subsets of positive operators in those ideals.
Here and below, ``operator ideals'' are understood in the sense of \cite{Pie}.
More precisely: the operator ideal ${\mathcal{I}}$ ``assigns'' to any pair of Banach
spaces $X$ and $Y$, a linear subspace ${\mathcal{I}}(X,Y) \subset B(X,Y)$, in such a way
that the following conditions are satisfied:
\begin{enumerate}
\item 
If $T \in B(X,Y)$ has finite rank, then $T \in {\mathcal{I}}(X,Y)$.
\item
If $T \in {\mathcal{I}}(X,Y)$, $U \in B(X_0,X)$, and $V \in B(Y,Y_0)$,
then $VTU \in {\mathcal{I}}(X_0,Y_0)$.
\end{enumerate}

Several ideals of operators will appear throughout this paper. The well-known ideals
of compact, respectively weakly compact, operators, are denoted by $\iK$ and $\iWK$.
$T \in B(X,Y)$ is called \emph{Dunford-Pettis} (DP for short) if it takes weakly
compact sets to relatively compact sets. Equivalently, it maps weakly null sequences
to norm null sequences (see e.g. \cite[Chapter 5]{AK} for more information).
$T$ is said to be \emph{finitely strictly singular} (FSS)  if for every
$\vr>0$ there exists $N \in \mathbb{N}$ such that every subspace of dimension at
least $N $ contains a normalized vector $x$ such that $\|Tx\| <\vr$.
$T$ is \emph{strictly singular} (SS) if it is not an isomorphism on any
infinite-dimensional subspace of $X$. $T$ is \emph{strictly cosingular} (SCS) if,
for any infinite-dimensional subspace $Z \subset Y$, the operator $Q_Z T$ is not
surjective ($Q_Z : Y \to Y/Z$ is the quotient map). And it is \emph{inessential} (IN)
or \emph{ Fredholm perturbation} if $I+UT$ is Fredholm for every operator $U:Y \to X$.
The set of all DP operators from $X$ to $Y$ is denoted by $\iDP(X,Y)$.
The notations $\iFSS$, $\iSS$, and $\iSCS$ have similar meaning.
It is known that $\iK(X,Y) \subseteq \iFSS(X,Y) \subseteq \iSS(X,Y) \subseteq \iIN(X,Y)$ and $ \iK(X,Y) \subseteq \iSCS(X,Y) \subseteq \iIN(X,Y)$.
%\begin{eqnarray*}
%\iK(X,Y) \subseteq \iFSS(X,Y) \subseteq \iSS(X,Y) \subseteq \iIN(X,Y),\\
%\iK(X,Y) \subseteq \iSCS(X,Y) \subseteq \iIN(X,Y),
%\end{eqnarray*}
In general, these inclusions are proper. The classes
$\iDP$, $\iFSS$, $\iSS$, $\iSCS$,
and $\iIN$ are closed operator ideals, in the sense of \cite{DJP}. 
For more information, the reader is referred to \cite{Ai}, \cite{Pie}.

We also use adjoint ideals: if ${\mathcal{I}}$ is an ideal, we define
${\mathcal{I}}^*(X,Y) = \{T \in B(X,Y) : T^* \in {\mathcal{I}}(Y^*,X^*)\}$.
By \cite[p.~394]{Ai}, $\iSS^*(X,Y) \subset \iSCS(X,Y)$, and
$\iSCS^*(X,Y) \subset \iSS(X,Y)$. In general, these inclusions are proper.

In the analysis of strict singularity of operators, the subprojectivity is
often used. A Banach space $X$ is said to be \emph{subprojective} if any
infinite dimensional subspace $Y \subset X$ contains a further subspace,
complemented in $X$. By \cite{Whi}, if $Y$ is subprojective, then
$\iSS^*(X,Y) \subset \iSS(X,Y)$.

Inclusions and equivalences of the ideals of operators
% mentioned above on function (sequence) spaces
were first investigated in \cite{Ca, GFM}, where it was shown that
$\iK(Z)$ is the only closed ideal in $B(Z)$ when $Z$ is $\ell_p$
($1 \leq p < \infty$) or $c_0$.
The more complicated setting of operators on $C(K)$ or $L_1$ spaces was
studied in \cite{Pe65-I, Pe65-II}.
Further work in this direction includes
\cite{We, Tra, Ru}, and, most recently, \cite{Lef}. In the present paper, we
investigate these ideals on $C^*$ algebras and their (pre)duals, either
extending or finding analogues to the classical results known for $C(K)$
and abstract $L_1$-spaces.

The paper is structured as follows. In Section \ref{s:prelim}, we recall the
necessary definitions and facts.
In Section \ref{s:SS}, we study the inclusions of various ideals of operators,
and use these inclusions to classify von Neumann algebras. In particular,
we prove that a von Neumann algebra $\A$ is of finite type $I$ if and only if
$\iFSS(\A)=\iSS(A)=\iIN(\A)=\iWK(\A)$, and otherwise, these classes are different
(Theorem \ref{thm:IN=WK}).

Section \ref{s:SS C*} is largely devoted to proving Theorem \ref{t:c_0,l_2 SS}:
a non-strictly singular operator on a $C^*$ algebra must fix either
a copy of $c_0$, or a complemented copy of $\ell_2$. In Section \ref{s:DPP},
this description of strictly singular operators is used to characterize
$C^*$-algebras and their preduals with the Dunford-Pettis Property
(Proposition \ref{p:predual DPP}, Corollary \ref{c:comp l2}).

In Section \ref{s:duality} we connect strict singularity of an operator
to that of its adjoint.  We show that, for a compact $C^*$-algebra $\A$,
$\iSS(\A) = \iSS^*(\A) = \iSCS(\A) = \iIN(\A)$
(Corollaries \ref{c:C* SS T=T*} and \ref{c:SS SCS IN}).
Further, we prove that for a von Neumann algebra $\A$, the following
are equivalent: (i) $\A$ is of type $I$ finite; (ii)
$\iWK(\A_*) \subset \iSS(\A_*)$; (iii) $\iWK(\A_*) \subset \iSS^*(\A_*)$
(Proposition \ref{p:WK SS* SS}).

In Section \ref{s:incl}, we describe inclusions between positive parts of
the sets of operator ideals, proving, among other things, that
a von Neumann algebra $\A$ is purely atomic iff
$\iIN(\A_*)_+ = \iSS(\A_*)_+ = \iSCS(\A_*)_+ = \iWK (\A_*)_+=\iK(\A_*)_{+}$
(Proposition \ref{c:id in vna*}).

Finally, in Section \ref{s:prod} we study products of strictly singular
operators. In particular, we show that a von Neumann algebra $\A$
is of finite type $I$ if and only if the product of any two strictly singular
operators on $\A$ is compact (Proposition \ref{p:prod SS}).

%In the companion paper \cite{OS_id}, we investigate ideals of operators
%on non-commutative $L_p$ spaces and Schatten spaces.
%Among other things, we obtain a description of strictly singular operators on
%non-commutative $L_p$ spaces, similar to Theorem \ref{t:c_0,l_2 SS}.

Throughout, we use the standard Banach space facts and notation.
The word ``subspace'' refers to a closed infinite dimensional
subspace, unless specified otherwise. We say that an operator
$T \in B(X,Y)$ \emph{fixes a} (\emph{complemented}) \emph{copy of $E$}
if $X$ contains a (complemented) subspace $Z$, isomorphic to $E$, so that
$T|_Z$ is an isomorphism. $\ball(X)$ stands for the closed unit ball of $X$.
Finally, we denote by $\cs_p(H)$ the Schatten $p$-space of operators on a
Hilbert space $H$. In particular,  the ideal of compact operators $\iK(H)$
corresponds to $\cs_\infty(H)$ and its dual  to $\cs_1(H)$. 

\section{Preliminaries}\label{s:prelim}

\subsection{Classes of $C^*$-algebras}\label{ss:classes}
We are especially interested in two classes of $C^*$-algebras --
the compact ones, and the scattered ones.

A Banach algebra $\A$ is called \emph{compact} if, for any $a \in \A$,
the operator $\A \to \A : b \mapsto aba$ is compact. A $C^*$-algebra
is compact if and only if  it is of the form $(\sum_{i \in I} \iK(H_i))_{c_0}$.
We refer to \cite{Al} for other properties
of compact $C^*$-algebras.

To define a larger class of scattered $C^*$-algebras,
consider a $C^*$-algebra $\A$, and
$f \in \A^*$. Let $e \in \A^{**}$ be its support
projection. Following \cite{Je}, we call $f$ \emph{atomic} if every non-zero
projection $e_1 \leq e$ dominates a minimal projection (all projections are
assumed to ``live'' in the enveloping algebra $\A^{**}$).
Equivalently, $f$ is a sum of pure positive functionals. We say that $\A$
is \emph{scattered} if every positive functional is atomic. By \cite{Hu, Je},
the following three statements are equivalent: (i) $\A$ is scattered;
(ii) $\A^{**} = (\sum_{i \in I} B(H_i))_\infty$; (iii) the
spectrum of any self-adjoint element of $\A$ is countable.

For a $C^*$-algebra $\A$, the following statements
are equivalent: (i) $\A$ is scattered; (ii) $\A^*$ has the Radon-Nikodym Property (RNP);
(iii) $\A$ contains no isomorphic copies of $\ell_1$; (iv) $\A$ contains no
isomorphic copies of $C[0,1]$. Indeed, (i) $\Leftrightarrow$ (ii) is established
in \cite{Chu}. Furthermore (see e.g. \cite{Chu} again), $\A^*$ having the RNP
is equivalent to
the dual of every separable subspace of $\A$ being separable. Thus,
(ii) $\Rightarrow$ (iii) $\vee$ (iv). If $\A$ is not scattered, then
$\A$ contains a self-adjoint element $a$ with uncountable spectrum $K \subset \R$.
Moreover, the $C^*$-algebra generated by $a$ is isomorphic to $C_0(K)$
(the space of continuous functions on $K$, vanishing at $0$).
$K$ is a compact metric space, hence, by Miliutin's Theorem (see e.g. \cite[Theorem 4.4.8]{AK}),
$C_0(K)$ contains a copy of $C[0,1]$. This yields (iii) $\Rightarrow$ (i).
Finally, (iv) $\Rightarrow$ (iii) is clear.

A different description of scattered $C^*$-algebras can be found in \cite{Ku}.

We say that a von Neumann algebra $\A$ is (\emph{purely}) \emph{atomic} if
every projection in it has a minimal subprojection.
Equivalently, $\A$ is isomorphic to $(\sum_{i \in I} B(H_i))_\infty$.
A von Neumann algebra is said to be \emph{of finite type $I$} if
it is a finite sum of type $I_n$ von Neumann algebras, where $n$ is a positive
integer. Finite type $I$ von Neumann algebras  form a subclass of the larger class of \emph{type $I$ finite}
algebras, which are of the form $(\sum_{n \in \N} \A_n)_\infty$, with $\A_n$
being of type $I_n$, for some $n \in \N$.

\subsection{Dunford-Pettis Property}\label{ss:DPP}

In this subsection, we establish some characterizations of the
Dunford-Pettis Property for $C^*$-algebras and their preduals.
Recall that a Banach space has the \emph{Dunford-Pettis Property}
(\emph{DPP} for short) if, for any weakly null sequences
$(x_n) \subset X$ and $(x_n^*) \subset X^*$,
$\lim_n \langle x_n^*, x_n \rangle = 0$. Equivalently,
any weakly compact operator from $X$ to any Banach space is
Dunford-Pettis. The reader is referred to e.g. \cite{AK}
for more information.

The importance of the DPP for our project lies in the following
simple well-known observation:
% \begin{lemma}\label{l:DPP SS}
suppose $X$ is a Banach space with the Dunford-Pettis Property.
Then any weakly compact operator from $X$ to any Banach space is
strictly singular.
% \end{lemma}

% The rest of this section is devoted by expanding the known descriptions
% of $C^*$-algebras and preduals of von Neumann algebras with the DPP.

By \cite{Bu, CI}, the predual of a von Neumann algebra $\A$ has the DPP
if and only if $\A$ is type $I$ finite. By \cite{CIW}, a $C^*$-algebra has the DPP
if and only if it has no infinite dimensional irreducible representations.
Further descriptions of $C^*$-algebras and their
preduals which possess the DPP can be found in Section \ref{s:DPP}
(see e.g. Proposition \ref{p:predual DPP} and Corollary \ref{c:comp l2}).

\section{Inclusions of operator ideals and classes of von Neumann algebras}\label{s:SS}

In this section, we % characterize the class of strictly singular operators and study their inclusion relation with the class of finitely strictly singular,  inessential and weakly compact operators, when the domain and/or range space is a $C^*$-algebra.
study the inclusions between various ideals of operators, and use these inclusions
to characterize classes of von Neumann algebras.

First, we observe that \cite[Corollary 6]{Pf} and \cite{Pe65-I}
immediately yield the following.

\begin{proposition}\label{p:SS subs WK}
For a $C^*$-algebra $\A$ and a Banach space $X$, $\iSS(\A,X) \subset \iWK(\A,X).$
\end{proposition}

\begin{proposition}\label{p:IN subs WK}
Let $\A$ be either  a separable  $C^{*}$-algebra, or a von Neumann algebra.
Then $\iIN(\A) \subset \iWK(\A)$.
\end{proposition}
\begin{proof}
It suffices to show that, for any $T \notin \iWK(\A)$, there exists an
infinite dimensional subspace $M$ such that $T(M)$ is complemented.
Indeed, then there exists $S : T(M) \to M$ so that $ST|_M = I_M$.
Denote by $P$ a projection onto $T(M)$, and note that $I-SPT$ is not Fredholm.

Since $T^{*}$ is not weakly compact, \cite[Theorem~1]{Pf} yields $\vr > 0$ and
a disjoint normalized sequence of self-adjoint elements $x_n \in A$, spanning a subspace isomorphic to  $c_0$,  such that
$\sup_{f\in \ball(X^*)}{|T^{*}f(x_n)|}>\vr$ for every $n$. In particular,
$T|_{\span[x_n:n \in \N]}$ is not weakly compact. By \cite[Theorem 5.5.3]{AK}
 there exists a subspace $E$ of $\span[x_n : n \in \N]$, isomorphic to $c_0$, so that $T|_E$
is an isomorphism. Using a gliding hump argument, we can assume that
$E = \span[y_m : m \in \N]$, where $(y_m)$ is a normalized block basis of $(x_n)$
(hence the operators $y_m$ are also disjoint, in the sense that
$y_m^* y_k = y_m y_k^* = 0$ if $k \neq m$).

If $\A$ is separable, then $M = E$ works, since $c_0$ is separably injective.
If $\A$  is a von Neumann algebra, consider the space $F \subset \A$ of operators
$\sum_m \omega_m y_m$, with $\sup_m |\omega_m| < \infty$. Then $F$ is isometric to
$\ell_\infty$, and $T|_F : F \to \A$ is not weakly compact. By \cite[Theorem~5.5.5]{AK},
$F$ contains a subspace $M \approx \ell_\infty$, so that $T|_M$ is an isomorphism.
The injectivity of $\ell_\infty$ yields  that $T(M)$ is complemented.
\end{proof}

\begin{remark}\label{r:IN not WK}
For general $C^*$-algebras $\A$, the inclusion $\iIN(\A) \subset \iWK(\A)$
does not hold, even in the commutative case.
Indeed, consider $\A = c_0 \oplus \ell_\infty$. Then the operator
$T = \begin{pmatrix} 0 & 0 \\ i & 0 \end{pmatrix} \in B(\A)$ ($i$ is an isometric
embedding of $c_0$ into $\ell_\infty$) is not weakly compact. However,
reasoning as in the proof of Theorem \ref{thm:IN=WK}, we conclude that $T \in \iIN(\A)$.
\end{remark}

This leads to a characterization of von Neumann algebras of finite type $I$
(that is, finite direct sums of von Neumann algebras of type $I_n$, with
finite $n$).

\begin{theorem}\label{thm:IN=WK}
A von Neumann algebra $\A$ is of finite type $I$ if and only if
$\iFSS(\A)=\iSS(A)=\iIN(\A)=\iWK(\A)$. Moreover, if $\A$ is not of finite type $I$,
then all of these classes are different.
\end{theorem}

For the proof we need a remark and a consequent lemma.

\begin{remark}\label{FSS-2-sum}
Every $2$-summing operator between Banach spaces is finitely strictly singular.
Indeed, let $X$ and $Y$  be Banach spaces and $T:X \to Y$ a $2$-summing operator. 
If $T$ is not $FSS$, then there exists $c>0$ and a sequence of
$n$-dimensional subspaces $E_n \subset X$ such that $\|Tx\| \geq c\|x\|$
for every $x\in E_n$.  Now recall that, for $E_n$, we have
$\pi_2(I_{E_n}) = \sqrt{n}$ \cite[Theorem 1.11]{Pis}.
Therefore, $ \pi_2(T) \geq \pi_2(T|E_n) \geq c \pi_2(I_{E_n}) = c n^{1/2}$
holds for every $n$, leading to a contradiction.
\end{remark}

\begin{lemma}\label{l:FSS=WK}
% \item 
$\iSS(C(K),Y)=\iFSS(C(K),Y)=\iWK(C(K),Y)$ for any Banach space $Y$ and any
compact Hausdorff topological space $K$.
\end{lemma}

\begin{proof}
By \cite[Section 5.5]{AK}, $\iSS(C(K),Y)=\iWK(C(K),Y)$.
By \cite[Theorem~15.2]{DJT}, any $T \in \iWK(C(K),Y)$ is a norm
limit of a sequence of operators $(T_n)$, which factor through $\ell_2$.
By Grothendieck Theorem, the operators $T_n$'s are $2$-summing, and, by Remark~\ref{FSS-2-sum},  are FSS.  As the ideal of FSS operators is closed, $T$ must be FSS as well.
\end{proof}

\begin{proof}[Proof of Theorem \ref{thm:IN=WK}]
If $\A$ is of finite type $I$, then, by \cite[Theorem~6.6.5]{KR2}, it is
Banach space isomorphic to a commutative von Neumann algebra $\mathcal{B}$.
Lemma \ref{l:FSS=WK} and Proposition \ref{p:IN subs WK}
yield  $\iFSS({\mathcal{B}})=\iSS({\mathcal{B}})=\iIN({\mathcal{B}})=\iWK({\mathcal{B}})$.

% Recall that $\A$ is finite type $I$ if it is a direct sum of finitely
% many algebras of type $I_n$, where $n$  is a positive integer. By
% \cite[Theorem~6.6.5]{KR2}, any type $I_n$ algebra is isomorphic
% to $M_n \otimes C$, where $C$ is a commutative von Neumann algebra.
% Therefore it is isomorphic to $L_\infty(\mu)$ which, together with
% Theorem~\ref{thm:FSS=WK}, imply $\iFSS(\A)=\iSS(\A)=\iIN(\A)=\iWK(\A)$.

If $\A$ is not of finite type $I$, then (see e.g. \cite{Pop})
there exists a complete isometry $J : B(\ell_2) \to \A$ (in fact,
$J$ and $J^{-1}$ are completely positive).
By Stinespring-Wittstock-Arveson-Paulsen Theorem, there exists a
complete contraction $S : \A \to B(\ell_2)$, so that
$S = J^{-1}$ on $J(B(\ell_2))$. Denote by $E_{ij}$ the matrix units
in $B(\ell_2)$, and consider the map $T$, taking $E_{1j}$ to $E_{kj}$
($k$ is the unique integer satisfying $2^{k-1} \leq j < 2^k$), and
$E_{ij}$ to $0$ for $i > 1$. Clearly, $T$ can be viewed as a ``formal identity''
from $\ell_2$ to $(\oplus_k {\ell_2^{2^{k-1}}})_{c_0}$, thus it is not
finitely strictly singular. Hence, $JTS \in \iSS(\A) \setminus \iFSS(\A)$.

Moreover, $B(\ell_2)$ contains a subspace $Z = \ell_2 \oplus_\infty \ell_\infty$,
complemented via a projection $P$.
Consider $T = \begin{pmatrix} 0 & 0 \\ i & 0 \end{pmatrix} \in B(Z)$, where $i$ is an isometric
embedding of $\ell_2$ into $\ell_\infty$. Clearly the operator $JTPS$
is not strictly singular. We shall show that $T \in \iIN(Z)$. By \cite{P}  it suffices to show that $I - AT$ has finite dimensional kernel for any
$A=\begin{pmatrix} A_1 & A_2 \\ A_3 & A_4 \end{pmatrix} \in B(Z)$. Indeed,  $\ker (I - AT)$
consists of all vectors $x \oplus y$ ($x \in \ell_2$, $y \in \ell_\infty$)
satisfying $x \in \ker (I - A_2 i)$, and $y = A_4 i x$.
By \cite[Section 5.5]{AK}, any operator from $\ell_\infty$ to a separable
Banach space is strictly singular.
%  Applying this result to $A_2$, we see that
Thus, $I - A_2 i$ is Fredholm, hence its kernel is finite dimensional.
Therefore, $\ker (I - AT)$ is finite dimensional.

% Finally, 
To finish the proof, note that there is a projection from $B(\ell_2)$ on a copy of $\ell_2$,
which is a weakly compact but, evidently, not an inessential operator.
\end{proof}
Combining Propositions \ref{p:IN subs WK} and \ref{l:FSS=WK} we obtain the following.
\begin{corollary}\label{c:IN=WK}
For a von Neumann algebra $\A$, the following are equivalent:\\
(1) $\A$ is not of finite type $I$, (2) $\iIN(\A)$ is a proper subset of $\iWK(\A)$.
%\begin{enumerate}
%\item
%$\A$ is not of finite type $I$.
%\item
%$\iIN(\A)$ is a proper subset of $\iWK(\A)$.
%\end{enumerate}
\end{corollary}

\begin{proposition}\label {p:cont in WK} 
Let $\A$ be a von Neumann algebra. Then the ideals $\iSS(X,\A_*)$,  $\iSCS(X,\A_*)$,
and $\iIN(X,\A_*)$ are subsets of $\iWK(X,\A_*)$.
\end{proposition}

\begin{proof}
By \cite[Example IV.1.1]{HWW},  $\A_*$ is $L$-embedded, and, therefore,
has property $(V^*)$ by \cite{Pf:93}. Thus,  if $T \not \in \iWK(\A_*)$,
then by \cite[Proposition~1.10]{Bo}, there exists a subspace $M \subset \A$
isomorphic to $\ell_1$ such that $T(M)$ is isomorphic to $\ell_1$ and complemented
in $\A_*$ by projection $P$. This implies that $T$ is not inessential,
since the dimension of $\ker(I -T(T^{-1}P))$ is infinite.  
\end{proof}

\section{Strictly singular operators on $C^*$-algebras}\label{s:SS C*}

In this section we describe strictly singular operators acting from
a $C^*$-algebra:
 
%We also prove an analogue of the main result of \cite{OS_id}.

\begin{theorem}\label{t:c_0,l_2 SS}
Suppose $\A$ is a $C^*$-algebra, and $X$ is a Banach space. For
$T \in B(\A,X)$, the following are equivalent:
\begin{enumerate}
\item $T$ is not strictly singular.
\item $T$ fixes either a copy of $c_0$, or a complemented copy of $\ell_2$.
\end{enumerate}
Moreover, if $\A$ is a von Neumann algebra, then $(1)$ and $(2)$
are equivalent to
\begin{enumerate}
\item[(2$^\prime$)] $T$ fixes either a complemented copy of $\ell_2$,
or a copy of $\ell_\infty$, complemented by a weak$^*$ continuous
projection.
\end{enumerate}
Furthermore, if $X$ is a dual space, and $T$ is weak$^*$
continuous, then a copy of $\ell_2$ mentioned in $(2^\prime)$ can be
chosen so that it is complemented by a weak$^*$-continuous projection.
\end{theorem}

\begin{remark}\label{r:c_0,l_2 SS}
% In contrast with Theorem \ref{t:main}, 
Similar results were obtained in \cite{OS_id} for operators on certain non-commutative
$L_p$ spaces. In contrast with those results, $T(E)$ need not be complemented, even
if $\A$ is separable. Indeed, let $P$ and $J$ be a projection from $\iK(\ell_2)$ onto
a copy of $\ell_2$, and an isometric embedding of $\ell_2$ into $C[0,1]$, respectively.
For $T = JP$, $T(E)$ is not complemented whenever $T|_E$ is an isomorphism.

The commutative counterpart of Theorem \ref{t:c_0,l_2 SS}
(see \cite[Section 5.5]{AK}, or \cite{Pe65-I, Pe65-II}) states that,
for $T \in B(C(K),X)$ ($K$ is a Hausdorff compact), the following are equivalent:
% \begin{enumerate}
% \item $T$ is not strictly singular.
% \item $T$ fixes a copy of $c_0$.
% \item $T$ is not weakly compact.
% \end{enumerate}
$(1)$ $T$ is not strictly singular;
$(2)$ $T$ fixes a copy of $c_0$;
$(3)$ $T$ is not weakly compact.
Moreover, if $C(K)$ is injective, the above three statements
are equivalent to
% \begin{enumerate}
% \item[(2$^\prime$)] $T$ fixes a copy of $\ell_\infty$.
% \end{enumerate}
$(2^\prime)$ $T$ fixes a copy of $c_0$.
\end{remark}

%The commutative version of Theorem \ref{t:c_0,l_2 SS} is known.
%By \cite[Section 5.5]{AK} (see also \cite{Pe65-I, Pe65-II}),
%for $T \in B(C(K),X)$ ($K$ is a Hausdorff compact), the following are equivalent:
%\begin{enumerate}
%\item $T$ is not strictly singular.
%\item $T$ fixes a copy of $c_0$.
%\item $T$ is not weakly compact.
%\end{enumerate}
%Moreover, if $C(K)$ is injective, the above three statements
%are equivalent to
%\begin{enumerate}
%\item[(2$^\prime$)] $T$ fixes a copy of $\ell_\infty$.
%\end{enumerate}

To prove Theorem \ref{t:c_0,l_2 SS}, we first establish:

\begin{lemma}\label{l:adjoint}
Suppose $\A$ is a von Neumann algebra, $X$ is a dual space, and $T \in B(\A,X)$
is a weak$^*$ continuous weakly compact operator. Then there there exist a weak$^*$ to weak$^*$ continuous
isometric embedding $J : X \to \ell_\infty(\Gamma)$, and, for every $\vr > 0$,
a Hilbert space $H$, and weak$^*$ to weak$^*$ continuous operators $U \in B(\A,H)$
and $V \in B(H,\ell_\infty(\Gamma))$, so that $\|JT - VU\| \leq 2 \vr$.
\end{lemma}

\begin{proof}
We modify the construction from \cite{Jar}.
For an appropriate index set $\Gamma$, there exists a quotient map
$q: \ell_1(\Gamma) \to X_*$. Then $q^* = j$ is a weak$^*$ continuous
isometric embedding of $X$ into $\ell_\infty(\Gamma)$.
Let $T_* \in B(X_*, \A_*)$ be a preadjoint of $T$.
Then $T_*(\ball(X_*))$ is a relatively weakly compact subset of $\A_*$.
By \cite[Chapter III]{Tak}, for every $\vr > 0$ there exist $\delta > 0$
and a positive $\omega \in \A_*$ so that
$\sup_{\phi \in T_*(\ball(X_*))} |\phi(a)| < \vr$ whenever
$\omega(a^* a + a a^*) < \delta$, $a \in \A$. Define the semi-inner product
$\langle a,b \rangle_\omega = \omega (a b^* + b^* a)$ ($a,b \in \A$),
and let $\| \cdot \|_\omega$ be the corresponding seminorm.
Furthermore, set $K = \{a \in \A : \|a\|_\omega = 0\}$,
and let $H$ be the completion of $\A/K$ in $\| \cdot \|_\omega$.

To proceed, observe that the formal identity $i : \A \to H$ is weak$^*$
continuous. Indeed, suppose $(a_\alpha)$ is a weak$^*$ null net in $\A$,
we need to show  that $(i a_\alpha)$ is weakly null in $H$. By a density
argument, we only need to prove that
$\lim_\alpha \omega(a_\alpha y) = \lim_\alpha \omega(y a_\alpha) = 0$
for any $y \in \A$. However, it is well known that the normality of
$\omega$ implies the normality of $\omega y$ and of $y \omega$.

% it suffices to show that $\ker i = K$ is a weak$^*$ closed
% subspace of $\A$. Note that $K = K_l \cap K_r$, where
% $K_l = \{a \in \A : \omega(a^* a) = 0\}$, and
% $K_r = \{a \in \A : \omega(a a^*) = 0\}$. By \cite[Proposition III.3.12]{Tak},
% $K_l$ and $K_r$ are closed in the weak$^*$ topology.

It is shown in \cite{Jar} that there exists $C = C_\vr$ so that
$\|T a\| \leq C \|i a\| + \vr \|a\|$ for any $a \in \A$.
This can be restated as follows: let ${\mathcal{B}} = H \oplus_1 \A$,
and $i^\prime : \A \to {\mathcal{B}} : a \mapsto C ia \oplus \vr a$. Then
$\|i^\prime a\| \geq \|Ta\|$. Clearly, $i^\prime$ is weak$^*$ continuous,
with weak$^*$ closed range ${\mathcal{B}}^\prime$.
Denote the preadjoint of $i^\prime$ by $i^\prime_*$.
% with preadjoint $i^\prime_*$.
Furthermore, let $j$ be the canonical embedding of ${\mathcal{B}}^\prime$
into ${\mathcal{B}}$, and let $j_*$ be its preadjoint (a quotient map).

Define $S : {\mathcal{B}}^\prime \to X : i^\prime a \mapsto T a$. It is easy to
see that $S$ is a weak$^*$-continuous contraction, with preadjoint $S_*$.
The operator $S_* q : \ell_1(\Gamma) \to {\mathcal{B}}^\prime_*$
has a lifting $\tilde{S}_* : \ell_1(\Gamma) \to {\mathcal{B}}_*$,
so that $\|\tilde{S}_*\| < 2$, and $S_* q = j_* \tilde{S}_*$. Then
$\tilde{S} = (\tilde{S}_*)^* : {\mathcal{B}} \to \ell_\infty(\Gamma)$
is a weak$^*$ continuous extension of $S$, of norm less than $c$.
Furthermore, $T = S i^\prime$, hence
$
q^* T = q^* S i^\prime = (S_* q)^* i^\prime = (j_* \tilde{S}_*)^* i^\prime =
\tilde{S} j i^\prime .
$

Denote by $P_H$ and $P_{\mathcal{A}}$ the canonical projections from ${\mathcal{B}}$
to $H$ and $\A$, respectively. Then, for  any $a \in \A$,
$\|q^* T a - \tilde{S} P_H i^\prime a\| \leq \|\tilde{S}\| \|P_{\A} i^\prime a\|
 \leq 2 \vr \|a\|$. Then $U = P_H i^\prime \in B(\A,H)$ and $V = \tilde{S}|_H$
are weak$^*$ continuous, and satisfy $\|q^* T - VU\| \leq 2 \vr$.
%
% We can write $\tilde{S} = \tilde{S}_H + \tilde{S}_\A$
% Denote by $\tilde{S}_H$ and $\tilde{S}_\A$ the restrictions of $\tilde{S}$
% to $H$ and $\A$, respectively. For any $a \in \A$,
% $\|q^* T a - \tilde{S}_H i a\| \leq \vr \|\tilde{S}\| \|a\| < c \vr \|a\|$.
\end{proof}

\begin{lemma}\label{l:C*alg WK}
Suppose $\A$ is a $C^*$-algebra, and a weakly compact $T \in B(\A,X)$
is an isomorphism on $Y \subset \A$. Then $Y$ is isomorphic to a Hilbert space,
and complemented in $\A$. If, furthermore, $\A$ is a von Neumann algebra, $X$ is
a dual space, and $T$ is weak$^*$ continuous, then $Y$ is complemented by a
weak$^*$ continuous projection.
\end{lemma}

\begin{proof}
% First consider the case when $\A$ is a $C^*$-algebra.
Suppose $\A$ is a $C^*$-algebra.
By enlarging the space $X$ if necessary, we can assume
$X = \ell_\infty(\Gamma)$, for some set $\Gamma$. By scaling,
it suffices to consider the case of $\|T\| = 1$. If $T$ fixes $Y \subset \A$,
denote the restriction of $T$ to $Y$ by $T_0$, and let $Y_0 = T(Y)$.
Find $c \in (0,1/3)$ so that
$\|T_0^{-1}\|_{B(Y_0,Y)} < 1/(3c)$. By \cite{Jar}, there exist a Hilbert space $H$,
and operators $U \in B(\A,H)$ and $V \in B(H,X)$, so that $\|T - VU\| < c$.
Note that, for any $y \in Y$, $\|VUy\| \geq 2c\|y\|$, hence $VU|_Y$ is an
isomorphism. Then $H_1 = U(Y)$ and $Y_1 = VU(Y) = V(H_1)$ are closed subspaces
of $H$ and $X$, respectively, while $U_1 = U|_Y \in B(Y,H_1)$ and
$V_1 = V|_{H_1} \in B(H_1,Y_1)$ are isomorphisms.

Now consider the map $S : Y_0 \to Y_1 : T y \mapsto VU y$ ($y \in Y$).
We claim that $S$ is an isomorphism. Indeed, pick a norm one $z \in Y_0$.
Then $y = T_0^{-1} z \in Y$ satisfies $1 \leq \|y\| \leq (3c)^{-1}$, hence
$\|(VU - T)y\| \leq \|VU-T\| \|y\| < 1/3$, and therefore,
$$
\Big| \|Sz\| - 1 \Big| \leq \|VU - T\| \|y\| < \frac{1}{3} .
$$
Thus, $Y_1$ is isomorphic to $Y_0$, hence also to $Y$. As
$V_1 U_1 \in B(Y, Y_1)$ is an isomorphism,
we conclude that $Y$ is a Hilbert space.

To show that $Y$ is complemented, set $R = VU$, and
let $Q$ be the orthogonal projection from $H$ onto $H_1$.
Note that % the operators $V_1 = V|_{H_1} \in B(H_1,Y_1)$ and
% $R_1 = R|_Y \in B(Y,Y_1)$ are isomorphisms.
$R_1 = R|_Y \in B(Y,Y_1)$ is an isomorphism.
Then $P = R_1^{-1} V Q U : \A \to Y$ is a projection.

Now suppose $T : \A \to X$ is a weak$^*$ to weak$^*$ continuous,
weakly compact contraction, which is not strictly singular.
Then $T$ fixes a reflexive subspace $Y \subset \A$. By \cite[Lemma 6.44]{FHetc},
$Y$ is weak$^*$ closed in $\A$. Pick $c > 0$ so that $\|Ta\| \geq 3c \|a\|$
for any $a \in Y$. By the above, $Y$ is isomorphic to a Hilbert space.
By Lemma \ref{l:adjoint}, we can find a weak$^*$ to weak$^*$  continuous
isometry $J : X \to \ell_\infty(\Gamma)$, a Hilbert space $H$, and weak$^*$
to weak$^*$ continuous operators $U : \A \to H$ and $V : H \to \ell_\infty(\Gamma)$,
so that $\|VU - JT\| < c$. Repeating the construction employed in the
$C^*$-algebra case, we define a projection $P = R_1^{-1} V Q U$
from $\A$ to $Y$. We know that $U$ is weak$^*$ to weak continuous,
while $R_1^{-1}$, $V$, and $Q$ are norm continuous, hence also weak
continuous. Therefore, $P$ is weak$^*$ to weak continuous, hence
weak$^*$ continuous.
\end{proof}

\begin{proof}[Proof of Theorem \ref{t:c_0,l_2 SS}]
Suppose first $T$ is not weakly compact. By \cite{Pf}, $T$ fixes a copy of $c_0$,
spanned by self-adjoint normalized elements $a_i \in \A$, with disjoint support.
From this, one can conclude (see e.g. \cite[Theorem 5.5.5]{AK}) that, if $\A$ is
a von Neumann algebra, % then it contains a sequence of disjointly supported
%norm one elements $a_i = a_i^*$, 
then  $T$ is an isomorphism on a copy of $\ell_\infty$ generated by a disjoint sequence of norm one self-adjoint elements  $a_i$'s (we use ${\mathcal{B}}$ for this copy of $\ell_\infty$).
Denote the (mutually orthogonal) support projections
of the elements $a_i$ by $p_i$. Then $u : a \mapsto \sum_i p_i a p_i$ is a
conditional expectation on $\A$. It is easy to see that, for any $\vr > 0$, there exists
a weak$^*$ continuous projection $u_i$ from $p_i \A p_i$ onto $\span[a_i]$.
Consequently, $v = (\sum_i u_i) u$ is a weak$^*$ continuous projection from
$\A$ onto ${\mathcal{B}}$, of norm not exceeding $1+\vr$.

For $T \in \iWK(\A,X) \backslash \iSS(\A,X)$, invoke Lemma \ref{l:C*alg WK}.
\end{proof}

In the dual setting, we obtain:

\begin{corollary}\label{c:l_2 l_1 SCS}
Suppose $X$ is a Hilbert space, and $\A$ is a von Neumann algebra.
For $T \in B(X, \A_*)$, the following are equivalent:
\begin{enumerate}
\item $T$ is not strictly cosingular.
\item There is a quotient map $q$ from $\A_*$ onto either $\ell_1$ or $\ell_2$,
so that $qT$ is an isomorphism.
\end{enumerate}
\end{corollary}

\begin{proof}
Only $(1) \Rightarrow (2)$ needs to be proved. Pick $T \in B(X, \A_*)$ which is
not strictly cosingular. If $T$ is not weakly compact, then \cite{Pf} $T^*$
fixes a copy of $c_0$ (or $\ell_\infty$), hence, by \cite[Proposition 1]{Pe65-II},
$T$ fixes a complemented copy of $\ell_1$. Now suppose $T$ is weakly compact, and
$q : \A_* \to Z$ is a quotient map so that $qT$ is an isomorphism ($Z$ is infinite
dimensional). Then $T^*$ fixes $q^*(Z^*)$. By Lemma \ref{l:C*alg WK}, $Z$ is
a Hilbert space.
\end{proof}

\section{Characterizations of the Dunford-Pettis Property}\label{s:DPP}

Now we use the results from the previous sections to characterize $C^*$-algebras
and their preduals possessing the  Dunford-Pettis Property (DPP). In the predual
setting, we obtain the following:

% Recall that a Banach space $X$ is said to have the DPP
% if, for any Banach space $Y$, $\iWK(X,Y) \subseteq \iDP(X,Y)$. The reader is
% referred to e.g. % \cite[Section 3.7]{MN} or
% \cite[Section 5]{AK} for more information.

\begin{proposition}\label{p:predual DPP}
For a von Neumann algebra $\A$, the following are equivalent:
\begin{enumerate}
\item 
$\A_*$ has the Dunford-Pettis Property.
\item
$\A$ is type $I$ finite.
\item
$\A_*$ does not contain complemented copies of $\ell_2$.
\end{enumerate}
\end{proposition}

\begin{proof}
$(1) \Leftrightarrow (2)$ is established in \cite{CI} and \cite{Bu}.
$(1) \Rightarrow (3)$ is evident. Now suppose $(2)$ fails.
If $\A$ is of type $I$, and not finite, then it contains a direct
summand of the form ${\mathcal{B}} \otimes B(H)$, where ${\mathcal{B}}$
is a commutative von Neumann algebra, and $H$ is infinite dimensional.
Consequently, $\A_*$ contains a complemented copy of $H$. If $\A$ is not
of type $I$, then, by \cite{MaN}, ${\mathcal{R}}_*$ is a complemented subspace
of $\A_*$, where ${\mathcal{R}}$ is a hyperfinite $II_1$ factor. However, ${\mathcal{R}}_*$
contains a complemented copy of $\ell_2$ (see e.g. \cite[Section 5]{LeMRR}; their
finite dimensional proof can be easily adjusted to the infinite dimensional case).
\end{proof}

% Theorem~\ref{t:c_0,l_2 SS} yields some inclusions of the set of Dunford-Pettis
% operators. It is easy to see that the canonical basis in either $c_0$ or $\ell_p$
To work with $C^*$-algebras, observe that the canonical basis in either $c_0$ or $\ell_p$
($1 < p < \infty$) forms a weakly null sequence which is not norm null. This leads to:

\begin{lemma}\label{l:not DP}
If $T \in B(X,Y)$ fixes a copy of either $\ell_p$ ($1 < p < \infty$) or $c_0$,
then $T$ is not Dunford-Pettis.
\end{lemma}

% Further, recall that, by \cite{Pf}, any non-weakly compact operator on a
% $C^*$-algebra fixes a copy of $c_0$.
% These facts are used to deduce two corollaries of Theorem~\ref{t:c_0,l_2 SS}:
This simple lemma is used to deduce two corollaries of Theorem~\ref{t:c_0,l_2 SS}:

\begin{corollary}\label{c:DP_WK}
Suppose $\A$ is a $C^*$-algebra and $X$ is a Banach space. Then
$\iDP(\A,X) \subseteq \iSS(\A,X) \subseteq \iWK(\A,X)$.
\end{corollary}

\begin{proof}
By \cite{Pf}, any non-weakly compact operator fixes a copy of $c_0$, hence
$\iSS(\A,X) \subseteq \iWK(\A,X)$. The inclusion $\iDP(\A,X) \subseteq \iSS(\A,X)$
is obtained by combining Theorem~\ref{t:c_0,l_2 SS} and Lemma \ref{l:not DP}.
\end{proof}

%Recall that a Banach space $X$ is said to have \emph{the Dunford-Pettis Property}
%if, for any Banach space $Y$, $\iWK(X,Y) \subseteq \iDP(X,Y)$. The reader is
%referred to e.g. % \cite[Section 3.7]{MN} or
%\cite[Section 5]{AK} for more information.

% We further obtain:
% Theorem~\ref{t:c_0,l_2 SS} yields:

\begin{corollary}\label{c:comp l2}
For a $C^*$-algebra $\A$, the following statements are equivalent:
\begin{enumerate}
\item 
$\A$ contains no complemented copies of $\ell_2$.
\item
$\A$ has no infinite dimensional irreducible representations.
\item
$\A$ has the Dunford-Pettis Property.
\item
For any Banach space $X$, $\iWK(\A,X) = \iSS(\A,X)$.
\item
For any Hilbert space $X$, $B(\A,X) = \iSS(\A,X)$.
\item
For any Banach space $X$, $\iSS(\A,X)=\iDP(\A,X)$.
\end{enumerate}
\end{corollary}

\begin{proof}
The implications $(1) \Rightarrow (2)$ and $(3) \Leftrightarrow (2)$ were established in
\cite{Rob} and \cite{CIW}, respectively. $(3) \Rightarrow (1)$ and
$(4) \Rightarrow (5) \Rightarrow (1)$ are clear. To prove $(1) \Rightarrow (4)$,
note first that $\iSS(\A,X) \subseteq \iWK(\A,X)$ by Proposition~\ref{p:SS subs WK}.
By Theorem~\ref{t:c_0,l_2 SS}, any $T \in \iWK(\A,X) \backslash \iSS(\A,X)$
must fix a complemented copy of $\ell_2$.

Now suppose the equivalent properties $(1)-(5)$ hold.
Then, by the definition of Dunford-Pettis Property,
$\iWK(\A,X) = \iSS(\A,X) \subseteq \iDP(\A,X)$. % To establish the inclusion
The inclusion $\iSS(\A,X) \supseteq \iDP(\A,X)$ follows from
Corollary \ref{c:DP_WK}.

% consider $T \in B(\A,X) \backslash \iSS(X,Y)$.
% By Theorem~\ref{t:c_0,l_2 SS}, $T$ fixes a copy of either $c_0$ or $\ell_2$.
% Both $c_0$ and $\ell_2$ contain weakly null sequences which are not norm null
% (take, for instance, the sequence consisting of the members of the canonical basis),
% hence $T$ cannot be strictly singular.

% Next prove that the equivalent properties $(1)-(5)$ imply $(6)$. Suppose (equivalently,
% all) properties $(1)-(5)$ hold, then there

Finally, to show $(6) \Rightarrow (1)$, note that, if (1) fails, then there exists a subspace
$E$ of $\A$, isomorphic to $\ell_2$, and complemented by a projection $Q$. Let $j : E \to \ell_2$
be an isomorphism, and let $i : \ell_2 \to \ell_1$ be the formal identity. Then
$T = ijQ \in B(\A, c_0)$ is strictly singular, but not Dunford-Pettis.
\end{proof}

We conclude this section with a related result.

\begin{proposition}\label{p:WK=IN}
For a Banach space $X$, and a $C^*$-algebra $\A$, the following are equivalent:
%(1) $\iWK(X,\A) \backslash \iIN(X,\A) \neq \emptyset$, (2) $X$ and $\A$ contain complemented copies of %$\ell_2$.
\begin{enumerate}
\item 
$\iWK(X,\A) \backslash \iIN(X,\A) \neq \emptyset$.
\item
$X$ and $\A$ contain complemented copies of $\ell_2$.
\end{enumerate}
\end{proposition}

\begin{proof}
$(2) \Rightarrow (1)$ is clear. To prove the converse, consider
$T \in \iWK(X,\A) \backslash \iIN(X,\A)$.
By \cite[Theorem 7.17]{Ai}, there exists $S \in B(\A,X)$ so that
$\dim \ker(I_{\A} - TS) = \infty$. In other words, there exists
an infinite dimensional subspace $E \subset \A$ so that $TS|_E = I_E$.
By Lemma \ref{l:C*alg WK}, % we can assume that
$E$ isomorphic to a Hilbert space, and
there exists a projection $P$ from $\A$ onto $E$. It is easy to check
that $S(E)$ is isomorphic to a Hilbert space, and complemented in $X$
via a projection $SPT$.
\end{proof}

\section{Strictly singularity of the adjoint operator}\label{s:duality}

In this section, we determine which conditions need to be imposed on an
operator $T$ to guarantee the strict singularity of $T^*$. First consider
von Neumann algebras and their preduals.

\begin{proposition}\label{p:SS SS* predual}
For a von Neumann algebra $\A$, the following are equivalent: \\
(1) $\iSS^*(\A_*) \subset \iSS(\A_*)$, (2) Either $\A$ is purely atomic, or $\A$ is type $I$ finite.
%\begin{enumerate}
%\item
%$\iSS^*(\A_*) \subset \iSS(\A_*)$.
%\item
%Either $\A$ is purely atomic, or $\A$ is type $I$ finite.
%\end{enumerate}
\end{proposition}

\begin{proof}
If $\A$ is purely atomic, then we can write it as $(\sum_i B(H_i))_\infty$.
In this case, $\A_* = (\sum_i \cs_1(H_i))_1$ is subprojective as $\ell_1$-sum of subprojective spaces, by \cite{OS_SP}.
By \cite[Theorem 2.2]{Whi}, $\iSS^*(\A_*) \subset \iSS(\A_*)$.
% Now suppose
If $\A$ is type $I$ finite, and $T \in \iSS^*(\A_*)$, then, by Proposition~\ref{p:SS subs WK},
$T$ is weakly compact. But $\A$ has the Dunford-Pettis Property, hence
$T$ is strictly singular.

Now suppose $T$ is neither purely atomic nor type $I$ finite.
By Proposition \ref{p:predual DPP}, $\A_*$ contains a subspace $E$
isomorphic to $\ell_2$, complemented by a projection $P$. Furthermore, by
\cite[Section 2.2]{OS}, $\A_*$ contains a subspace $F$, isomorphic to $L_1(0,1)$.
% and complemented by a projection $Q$.
Let $u : E \to F$ be an isomorphism.
Then $T = u P$ is not strictly singular. However, $T^*$ factors through $u^*$,
and the latter is $2$-summing, hence finitely strictly singular by Remark~\ref{FSS-2-sum}.
\end{proof}

\begin{proposition}\label{p:WK SS* SS}
For a von Neumann algebra $\A$, the following are equivalent:\\
(1) $\A$ is of type $I$ finite, (2) $\iWK(\A_*) \subset \iSS(\A_*)$,  (3) $\iWK(\A_*) \subset \iSS^*(\A_*)$.
%\begin{enumerate}
%\item 
%$\A$ is of type $I$ finite.
%\item
%$\iWK(\A_*) \subset \iSS(\A_*)$.
%\item
%$\iWK(\A_*) \subset \iSS^*(\A_*)$.
%\end{enumerate}
\end{proposition}

\begin{proof}
$(1) \Rightarrow (2)$ is an immediate consequence of $\A_*$ having the DPP
(by \cite{Bu, CI}, $\A_*$ has the DPP if and only if  $\A$ (1) holds).
To show $\lnot (1) \Rightarrow \lnot (2) \wedge \lnot (3)$, note that, if
$\A$ is not type $I$ finite, then, by the proof of Proposition \ref{p:predual DPP},
$\A_*$ contains a complemented copy of $\ell_2$.
A projection onto this copy of $\ell_2$
produces $T \in B(\A_*)$ so that neither $T$ nor $T^*$ are strictly singular.
To show $\lnot (3) \Rightarrow \lnot (1)$, consider
$T \in \iWK(\A_*) \backslash \iSS^*(\A_*)$. By Lemma \ref{l:C*alg WK},
there exists a weak$^*$ closed $Y \subset \A$, isomorphic to $\ell_2$ and
complemented by a weak$^*$ continuous projection, % $(P_*)^*$,
so that $T^*|_Y$ is an isomorphism.
Thus, in particular, $\A_*$ contains a complemented copy of $\ell_2$,
contradicting the DPP.
\end{proof}

Now turn to $C^*$-algebras.

\begin{proposition}\label{p:C* DP sep}
% Suppose $\A$ is a separable $C^*$-algebra with the Dunford-Pettis Property, and $T \in B(\A)$.
% If $T^*$ is strictly singular, then $T$ is strictly singular as well.
If $\A$ is a separable $C^*$-algebra with the Dunford-Pettis Property, then
$\iSS^*(\A) \subset \iSS(\A)$.
\end{proposition}

\begin{proof}
Consider $T \in B(\A) \backslash \iSS(\A)$. As $\A$ cannot contain complemented copies of $\ell_2$,
Theorem~\ref{t:c_0,l_2 SS} implies that $T$ fixes a subspace $E$, isomorphic to $c_0$.
By Sobczyk's Theorem, $T(E)$ is complemented in $\A$. Reasoning as in the proof of
\cite[Theorem 2.2]{Whi}, we conclude that $T^*$ is not strictly singular.
\end{proof}

\begin{proposition}\label{p:C* no DP}
If a non-scattered $C^*$-algebra $\A$ fails the Dunford-Pettis Property, then there exists a
non-strictly singular $T \in B(\A)$ with finitely strictly singular dual.
\end{proposition}

\begin{proof}%[Sketch of the proof]
By Corollary \ref{c:DP_WK}, $\A$ contains a subspace $E$, complemented via
a projection $P$, and isomorphic to $\ell_2$. Furthermore, by \cite{Ku},
$C(\Delta)$ embeds into $\A$ as a $C^*$-subalgebra (here $\Delta$ is the Cantor set).
Let $i : E \to C(\Delta) \subseteq \A$ be an isometry. Then $T = iP$ is not strictly
singular. However, $C(\Delta)^*$ is an abstract $L_1$-space, and therefore,
by Grothendieck's Theorem, $i^*$ is $2$-summing, hence finitely strictly singular
by Remark~\ref{FSS-2-sum}. Thus, $T^* = P^* i^*$ is finitely strictly singular.
\end{proof}

Propositions \ref{p:C* DP sep} and \ref{p:C* no DP} yield:

\begin{corollary}\label{c:sep C* T->T*}
Suppose $\A$ is a separable non-scattered $C^*$-algebra. Then $\A$ has
the Dunford-Pettis Property if and only if $\iSS^*(\A) \subset \iSS(\A)$.
% For a separable non-scattered $C^*$-algebra $\A$, the following are equivalent:
% \begin{enumerate}
% \item
% $\A$ fails the Dunford-Pettis Property.
% \item
% There exists a non-strictly singular $T \in B(\A)$ so that $T^*$ is strictly
% singular.
% \end{enumerate}
\end{corollary}

For some $C^*$-algebras $\A$, we can establish the equality $\iSS^*(\A) = \iSS(\A)$.

\begin{proposition}\label{p:duality}
For any Hilbert space $H$, $\iSS(\cs_\infty(H)) = \iSS^*(\cs_\infty(H))$.
% $T \in B(\cs_\infty(H))$ is strictly singular if and only if
% $T^*$ is strictly singular.
\end{proposition}
\begin{proof}
By \cite{Fr}, $\cs_\infty(H)$ is subprojective, hence,
by \cite[Theorem~2.2]{Whi}, the strict singularity of $T^*$ implies
the strict singularity of $T$. To prove the converse,
suppose, for the sake of contradiction, that $T$ is strictly singular,
but $T^*$ is not. Then there exists an infinite dimensional $X \subset \cs_1(H)$
so that $\|T^{*}x\| \ge c\|x\|$ for every $x\in X$ (here $c > 0$).
By \cite{Fr}, $X$ contains either $\ell_1$, or $\ell_2$.
By Proposition~\ref{p:SS subs WK}, $T$ is weakly compact, hence so is $T^{*}$.
Thus, by passing to a subspace, if necessary, we can assume $X\approx \ell_2$.
Then $T^*(X)$ is also isomorphic to $\ell_2$.

%For a projection $P$ on a Hilbert space $H$, define the operator
%$
%\Phi_P : \cs_\infty(H) \to \cs_\infty(H) : a \mapsto a - P^\perp a P^\perp ,
%{\textrm{  where  }}  P^\perp = I - P  .
%$
%For $b \in \cs_1(H)$, $\Phi_P^* b = b - P^\perp b P^\perp$.
%By \cite[Theorem~2]{Fr}, there exists a finite rank projection
%$P$ on $H$ so that $\Phi_P^*$ is an isomorphism on $T^* X$.
%Note that on the range of $\Phi_P^*$ all Schatten norms $\| \cdot \|_p$ are equivalent,
By \cite[Theorem~2]{Fr} there exist two finite rank projections $E$ and $F$ such that the operator
$
Q_{EF} : \cs_1(H) \to \cs_1(H) : a \mapsto a - E^\perp a F^\perp$ ,
where  $E^\perp = I -E$  and $ F^\perp=I-F$, is an isomorphism when restricted to $T^*X$. Note that on the range 
of $Q_{EF}$ all Schatten norms $\| \cdot \|_p$ are equivalent,  see the proof of \cite[Proposition~1]{Fr} for details.
This, together with the fact that  $Q_{EF}$ is also  bounded as an operator from $\cs_\infty(H)$ to $\cs_\infty(H)$, implies that,  for every $z \in T^* X$,
% $z \in X \cup T^* X$,
$$\|z\|_\infty \le \|z\|_1 \leq c_1 \|Q_{EF} z\|_1 \leq  c_1c_2 \|Q_{EF} z\|_\infty \leq 2 c_1 c_2 \|z\|_\infty,$$ for some $c_1,c_2>0$.
 Set $c_0=(2c_1c_2)^{-1}$.
% see \cite[Proposition~2 and the proof of Proposition~1]{Fr}.
Now consider the space $Y = (J (T^* X))^\star \subset \cs_\infty$ (here and
below, $\star$ stands for taking the adjoint of an operator in $B(H)$,
and $J$ is the formal identity from $\cs_1(H)$ to $\cs_\infty(H)$).
We claim that $T$ is an isomorphism on $Y$.
Indeed, pick $y \in Y$, with $\|y\|_\infty = 1$.
Then $\|J^{-1} y\|_1 \leq c_0^{-1}$, and consequently,
$x = c c_0 (T^*)^{-1} J^{-1} y^\star$ satisfies $\|x\|_1 \leq 1$.
Then
$
\|Ty\|_\infty \geq \tr((Ty) x) = \tr(y (T^*x)) = c c_0 \tr(y (J^{-1} y^\star)) =
c c_0 \|y\|_2^2 \geq c c_0^3 .
$
\end{proof}

% Before providing the proof, we establish two corollaries.

\begin{corollary}\label{c:C* SS T=T*}
% An operator on a compact $C^*$-algebra is strictly singular if and only if its
% adjoint is strictly singular.
If $\A$ is a compact $C^*$-algebra, then $\iSS(\A) = \iSS^*(\A)$.
\end{corollary}

\begin{proof}
Suppose $\A$ is a compact $C^*$-algebra and $T\in B(\A)$. From the representation of $\A$ mentioned in Section~\ref{ss:classes}, there exists a Hilbert space $H$
so that $\A$ is complemented in $\cs_\infty(H)$, via a projection $P$.
The corollary now follows from Proposition \ref{p:duality} and the fact that $T$ is strictly singular if and only if  $TP$ is. % Lemma \ref{l:proj} below. 
\end{proof}

%\begin{lemma}\label{l:proj}
%Suppose $X$ and $Y$ are Banach spaces, and $Z$ is a subspace of $X$,
%complemented by a projection $P$. Then $T \in B(Z,Y)$ is strictly singular
%if and only if $TP$ is strictly singular.
%\end{lemma}

\begin{corollary}\label{c:SS SCS IN}
If $\A$ is a compact $C^*$-algebra, then $\iSS(\A) = \iSCS(\A) = \iIN(\A)$.
\end{corollary}

\begin{proof}
By Corollary \ref{c:C* SS T=T*}, $\iSS(\A) = \iSS^*(\A)$.
By \cite{Fr}, $\cs_\infty(H)$ is subprojective for every $H$, hence so is $\A$.
Therefore, by \cite[Theorem 7.51]{Ai}, $\iSS(\A) = \iIN(\A)$.
By \cite[Theorem 7.44]{Ai}, $\iSCS(\A) \subset \iIN(\A)$. Finally,
$\iSCS(\A) \supset \iSS^*(\A)$, by \cite[Theorem 7.53]{Ai}.
\end{proof}

\begin{remark}\label{r:diff id}
In the dual of a compact $C^*$-algebra, the ideals listed in
Corollary \ref{c:SS SCS IN} need not coincide:
% Note that the ideal of cosingular operators acting on $\cs_1$ sits properly
% between the ideals of compact and strictly singular  operators.
% Indeed,  the  strictly cosingular operators are contained by the ideal of
% inessential operators. 
% Thus,  Corollary~\ref{c:main} yields
we have
$
\iK(\cs_1) \subsetneq \iSS^*(\cs_1) \subseteq \iSCS(\cs_1) \subsetneq
 \iSS(\cs_1) = \iIN(\cs_1) .
$
Indeed, in \cite{OS_id} we proved $\iIN(\cs_1) = \iSS(\cs_1)$.
By \cite[Chapter 7]{Ai} (or \cite{Whi}),
$\iSS^* \subseteq \iSCS \subseteq \iIN$.
By \cite[Theorem 2.3.1]{AK},
there exists a surjective operator  $T:X \to Y$, where $X$ and $Y$
are complemented subspaces of $\cs_1$ isomorphic to $\ell_1$ and $\ell_2$, respectively.
Clearly, $T$ is a strictly singular operator. But, being surjective,
$T$ is not strictly cosingular. This implies that
$S=TP \in \iSS(\cs_1) \setminus \iSCS(\cs_1)$, where $P$ is a projection from
$\cs_1$ onto $X$. On other other hand, let $i$ be the formal identity
from $X$ to $Y$. Then $U=iP$ is not compact, but $U^* = P^* i^*$ is strictly
singular (due to the fact that $i^*$ is $2$-summing).
\end{remark}

% \begin{remark}\label{r:T on C1 SS}
% Combining \cite[Theorem~2.2]{Whi} and \cite{Fr}, we conclude that
% $T^{*} \in \iSS( B(\ell_2))$ implies $T \in \iSS(\cs_1)$.
% In contrast with Proposition \ref{p:duality}, the converse is false.
% Indeed, let $j$ be a complemented embedding of $\ell_2$ into $\cs_1$,
% $p$ a projection from $\cs_1$ onto the diagonally embedded $\ell_1$, and
% $q$ a quotient map from $\ell_1$ to $\ell_2$. By Grothendieck's Theorem,
% $q$ is $1$-summing, hence (as already noted above) strictly singular.
% Thus, $jqp$ is strictly singular. However, $(jqp)^* = p^* q^* j^*$ is not
% strictly singular. Indeed, $q^*$ is an isometric embedding of $\ell_2$ into $\ell_\infty$,
% $p^*$ is an isometric embedding of $\ell_\infty$ into $B(\ell_2)$, and
% $j^* : B(\ell_2) \to \ell_2$ is a projection.
% 
% However, Proposition~ \ref{c:id in vna*} shows that any $T \in \iSS(\cs_1)_+$
% is compact, hence $T^*$ is also compact, and, consequently, strictly singular.
% \end{remark}

\section{ Positive parts of operator ideals}\label{s:incl}

In this section, we consider the inclusions between `positive parts' of operator ideals
(strictly singular, weakly compact, etc.).

 Recall that the positive cone $X_+$ in an ordered real Banach space $X$ is called \emph{generating} if $X = X_+ - X_+$.
In the complex case, we assume that $X_+$ is generating in $\Re X$ (the real part of $X$).
% (hence $X$ is the complexification of $X_+ - X_+$).
Equivalently,
there exists a constant ${\mathbf{G}}$ so that any $x$ in (the real part of) $X$ can be
written as $x = a - b$, with $a, b \in X_+$, and $\|a\| + \|b\| \leq {\mathbf{G}} \|x\|$.
Examples of spaces with generating cones include Banach lattices, $C^*$-algebras,
von Neumann algebra preduals, non-commutative function spaces, $JB$ and $JB^*$-algebras. The reader is referred to \cite{OS}, and references therein, for more information.

\begin{proposition}\label{c:id in vna*}
(1) If a von Neumann algebra $\A$ is purely atomic and $X$ is an ordered Banach space with a generating cone, then
$\iIN(X, \A_*)_+ = \iSS(X, \A_*)_+ = \iSCS(X, \A_*)_+ =
 \iWK (X, \A_*)_+=\iK(X, \A_*)_{+}$.

(2) If a von Neumann algebra $\A$ is not purely atomic, then
$\big( \iIN(\A_*)_+ \cap \iSS(\A_*)_+ \cap \iSCS(\A_*)_+ \cap \iWK (\A_*)_+ \big)
\backslash \iK(\A_*)_{+}$ is non-empty.
\end{proposition}

\begin{proof}
(1)  By Proposition~\ref{p:cont in WK},  it suffices to show that  every $T \in \iWK(X, \A_*)_+$  is  compact.  The later follows from \cite[Theorem~1.5.1]{OS}.  

To show (2), 
suppose $\A$ is not purely atomic. In \cite[Section~2.2]{OS} it is shown that $\A_*$ contains
a copy of $L_1(\Delta)$ ($\Delta = \{-1,1\}^\N$ is the Cantor set),
complemented by a positive projection.
% It therefore suffices to produce
% a positive operator on $L_1(\Delta)$, which is strictly singular, strictly
% cosingular, and weakly compact, but not compact.
Let $\one$ be the identity
function on $\Delta$, and let $r_i$ ($i \in \N$) be the $i$-th coordinate
function (in other words, the $r_i$'s are independent Rademacher functions).
Furthermore, $L_1(\Delta)$ contains a copy of $\ell_1$, generated by disjointly
supported norm one functions $e_i$, and positively complemented. We claim that
the positive operator $T : \ell_1 \to L_1(\Delta) : e_i \mapsto r_i + \one$
is strictly singular, strictly cosingular, and weakly compact, but not compact.
To this end, consider $S : \ell_1 \to L_1(\Delta) : e_i \mapsto r_i$.
Then $T - S$ has rank $1$. The sequence $(r_i)$ is equivalent to the $L_2$ basis,
hence we can think of $S$ as the canonical embedding $i : \ell_1 \to \ell_2$.
This operator is is weakly compact and strictly singular \cite{Mil}. Furthermore,
$i^* : \ell_\infty \to \ell_2$ is strictly singular, hence $i$ is strictly
cosingular.
\end{proof}

%Considering positive operators on $C^*$-algebras, we prove:

\begin{proposition}\label{p:id in cs_infty}
If $\A$ is a compact $C^*$-algebra, then
\\
$\iIN(\A)_+ = \iSS(\A)_+ = \iSCS(\A)_+ = \iWK (\A)_+=\iK(\A)_{+}$.
% $\iIN(\cs_\infty)_+ = \iSS(\cs_\infty)_+ = \iSCS(\cs_\infty)_+ =
%  \iWK (\cs_\infty)_+=\iK(\cs_\infty)_{+}$.
\end{proposition}

\begin{proof}
By Proposition~\ref{c:id in vna*},
$\iIN^*(\A)_+ = \iSS^*(\A)_+ = \iSCS^*(\A)_+ = \iWK (\A)_+=\iK(\A)_{+}$.
% By \cite[Theorem 7.53]{Ai}, $\iSS^* \subset \iSCS$, and $\iSCS^* \subset \iSS$.
By Corollaries \ref{c:C* SS T=T*} and  \ref{c:SS SCS IN},
$\iSS^*(\A) = \iSS(\A) = \iSCS(\A) = \iIN(\A)$.
\end{proof}

% For some $C^*$-algebras, the equality of the ideals from
% Proposition \ref{p:id in cs_infty} fails.

The situation is different for non-scattered $C^*$-algebras.

\begin{proposition}\label{p:scatter}
% For a $C^*$-algebra $\A$, the following statements are equivalent:
% \begin{enumerate}
% \item $\A$ is scattered.
% \item $\iWK(\A)_+ = \iK(\A)_+$.
% \end{enumerate}
If $\A$ is a non-scattered $C^*$-algebra, then
$\iK(\A)_{+}$ is a proper subset of
$\iSS(\A)_+ \cap \iSCS(\A)_+ \cap \iWK (\A)_+$.
\end{proposition}

\begin{proof}
Our goal is to exhibit $T \in B(\A)_+$ which is strictly singular, strictly cosingular, and
weakly compact, but not compact. To this end, we recall a construction from \cite[Section 2]{OS}.
As shown in that paper, there exist $y \in \A_+$, and normalized elements
$y_1, y_2, \ldots \in [0,y]$ with disjoint supports.
%  Thus, $B$
% contains an order bounded disjoint sequence, not convergent to $0$.
% Thus,
% there exists a disjoint normalized sequence $(y_i) \subset [0,y] \subset B$
% such that $\sum_{i=1}^{N}y_i<y$ for every $N \in \mathbb{N}$.
Furthermore, there exist $\psi \in \A^\star_{+}$, and a sequence
$(\phi_i) \subset [0, \psi]$ which has no norm convergent subsequences, but
which converges to $\phi$ weak$^*$. We claim that the operator $T$, defined by
\begin{displaymath}
Tx=\phi(x)y+\sum_{n=1}^{\infty}{(\phi_n-\phi)(x)y_n} \, \, (x \in \A)
\end{displaymath}
has the desired properties. Note that $\span[y_n : n \in \N]$ is isometric to $c_0$.
Consequently, the sum in the centered expression above converges in norm, since
$(\phi_n -\phi)(x) \to 0$ for any $x$.

Clearly, $T$ is positive. It is shown in \cite[Section 2]{OS} that $T$ is not compact.
However, there exists a rank $1$ operator $S$ so that $0 \leq T \leq S$.
 By \cite[Section 2]{OS} again, any operator on $\A$ dominated by a weakly compact operator
is weakly compact, hence $T$ must be weakly compact. Moreover, the image of $T$ lies in
$Y = \span[y, y_1, y_2, \ldots]$, which is isomorphic to $c_0$ (or $c$). This implies
the strict singularity of $T$. Indeed, otherwise there would exist a subspace $E \subset \A$,
so that $T|_E$ is an isomorphism. But $T(E)$ is a subspace of $c_0$, hence it must contain
a copy of $c_0$, contradicting the weak compactness of $T$. To show the strict cosingularity
of $T$, note that, by Eberlein-Smulian Theorem, $T(\ball(\A))$ is weakly sequentially compact.
Now apply \cite[Proposition 5]{Pe65-I}.
\end{proof}

% At the moment, we do not know whether the proper inclusion from
% Proposition \ref{p:scatter} characterizes non-scattered $C^*$-algebras.

\section{Products of strictly singular operators}\label{s:prod}

We apply the results of this paper to generalize a theorem from \cite{Mil}
to a non-commutative setting. For von Neumann algebras, we obtain:

\begin{proposition}\label{p:prod SS}
A von Neumann algebra $\A$ is of finite type $I$ if and only if
the product of any two strictly singular operators on $\A$ is compact.
\end{proposition}

\begin{proof}%[Sketch of a proof]
If $\A$ is of finite type $I$, then it is Banach isomorphic to a $C(K)$ space.
By \cite[Chapter 5]{AK}, $\iSS(\A) = \iWK(\A)$. Moreover, $\A$ has the
Dunford-Pettis Property, hence the product of any two weakly compact
operators is compact.

Now suppose $\A$ is not of finite type $I$. As noted in the proof of
Theorem \ref{thm:IN=WK}, $\A$ contains a complemented copy of $B(\ell_2)$.
Consequently, there exists a surjective isometry $i_2 : \ell_2 \to E$, and a contractive
projection $P$ from $\A$ onto $E$. Furthermore, $\A$ contains a copy of $\ell_\infty$,
hence also a copy of $\ell_1$. By \cite[Theorem 4.16]{DJT}, there exists a $2$-summing
quotient $q : \A \to \ell_2$. Let $i$ be the formal embedding of $\ell_2$
into $\ell_\infty$, and let $i_\infty$ be an isometric embedding of $\ell_\infty$ into $\A$.
Note that $q$ and $i$ are finitely strictly singular, 
the former due to it being $2$-summing, and the latter by \cite[Lemma 2]{Mil}.
Let $T = i_\infty i i_2^{-1} P$ and $S = i_2 q$. Then $TS$ is not compact.
\end{proof}

For $C^*$-algebras, the situation is not so clear.

\begin{proposition}\label{p:C* prod SS}
Suppose $\A$ is a $C^*$-algebra.
\begin{enumerate}
 \item
If $\A$ has the Dunford-Pettis property, then the product of any two strictly
singular operators on $\A$ is compact.
\item If $\A$ is a compact algebra,  then the product of any two strictly
singular operators on $\A$ is compact.
\item
If $\A$ fails the Dunford-Pettis Property, and is not scattered, then there
exist $S, T \in \iSS(\A)$ so that $TS$ is not compact.
\end{enumerate}
\end{proposition}

\begin{lemma}\label{l:milman}
If $\A$ is a compact $C^*$-algebra, then for any $T, S \in \iSS(\A^*)$,
$TS$ is compact.
\end{lemma}

%\query{we can remove this, and refer to \cite{OS_id}}
\begin{proof}
% From \cite[Theorem 1]{AL} we can easily deduce  that every basic sequence in $\A^*$
% contains either $\ell_1$ or $\ell_2$-subsequence.
Suppose, for the sake of contradiction, that $TS$ is not compact.
Then there exists a sequence $(x_i) \subset \ball(\A^*)$ so that
$\inf_{i>1} \dist (TS x_i, \span[ TS x_j : j < i ]) > 0.$
Applying Rosenthal's $\ell_1$ theorem
(see e.g. \cite[Theorem 10.2.1]{AK}) and passing to a subsequence,
we can assume that, for each of the sequences $(x_i)$, $(Sx_i)$, and
$(TSx_i)$, one of the following is true: (i) the sequence is weakly Cauchy, or
(ii) the sequence is equivalent to the $\ell_1$-basis.

Suppose first $(x_i)$ is weakly Cauchy. Then we can assume
that $(x_i)$ is weakly null. Indeed, let $x_i^\prime = (x_{2i} - x_{2i-1})/2$,
and observe that
$\inf_{i>1} \dist (TS x_i^\prime, \span[ TS x_j^\prime : j < i ]) > 0.$
Consequently, the sequences $(Sx_i)$ and
$(TSx_i)$ are weakly null as well. Passing to a further subsequence,
we can assume that all three sequences are basic. By passing
to a subsequence, and invoking \cite[Theorem 1]{AL}, we can assume that
all three sequences listed above are equivalent to the $\ell_2$-basis,
which contradicts the strict singularity of $T$ and $S$.
% Note that,
% for any weakly null sequence $(x_i) \subset \cs_1$, and any sequence
% $(\vr_j)$ of positive numbers, there
% exist $i_1 < i_2 < \ldots$ and $N_0 < N_1 < \ldots$, so that
% $\|x_{i_j} - (Q_{N_j} - Q_{N_{j-1}}) x_{i_j})\| < \vr_j$.
% Therefore, by passing to a subsequence, we can apply
% Theorem 2.4 (or Corollary 2.8) of \cite{Ar:81}, giving us that
% $(x_i)$ is equivalent to a sequence
% $(\alpha_i^{(1)} e_i^{(1)} \oplus \alpha_i^{(2)} e_i^{(2)}) \subset \ell_1 \oplus \ell_2$,
% where $(e_i^{(1)})$ and $(e_i^{(2)})$ are the canonical bases in $\ell_1$ and
% $\ell_2$ respectively. One more passage to subsequence, and we conclude that
% each of the sequences $(x_i)$, $(Sx_i)$, and $(TSx_i)$ is equivalent to the
% $\ell_2$ basis. This, however, contradicts the strict singularity of the operators
% involved.

Now suppose $(x_i)$ is equivalent to the $\ell_1$ basis. Due to the strict
singularity of $S$, $(Sx_i)$ cannot be equivalent to the $\ell_1$-basis,
hence it is weakly Cauchy. Moreover, we can assume that $(Sx_i)$ is weakly
null: as before, we pass to the sequence $x_i^\prime = (x_{2i} - x_{2i-1})/2$,
which is equivalent to the $\ell_1$-basis. Then $(TSx_i)$ is weakly null as well.
Passing to a subsequence as before, we obtain that both $(Sx_i)$ and $(TSx_i)$
are equivalent to the $\ell_2$-basis, contradicting the strict singularity of $T$.
\end{proof}

\begin{proof}[Proof of Proposition \ref{p:C* prod SS}]
(1)
The product of two weakly compact operators on a space with the DPP
is compact. By \cite{Pf}, any strictly singular operator on a $C^*$-algebra is weakly compact.
% hence the product of any two strictly singular

(2) Let $T$ and $S$ be strictly singular operators on
$\A=(\oplus_i \cs_\infty(H_i))_{c_0}$.
By Corollary \ref{c:C* SS T=T*},
$T^*$ and $S^*$ are strictly singular on $\A^*$. % $= (\oplus_i \cs_1(H_i))_{\ell_1}$.
By Lemma~\ref{l:milman}, $S^*T^*$ is compact.
This implies the compactness of $TS$.

(3) Suppose $\A$ is not scattered, and has infinite dimensional
irreducible representations.
As in the proof of Proposition \ref{p:prod SS}, we need to construct
$S, T \in \iSS(\A)$ so that $TS$ is not compact. As noted in
Section \ref{s:prelim}, $\A$ contains a copy of $C[0,1]$, hence also
a copy of $\ell_1$. By \cite[Theorem 4.16]{DJT}, we can find
a $2$-summing quotient map $q : \A \to \ell_2$. By \cite{Rob}, there exists an
isometry $i_2$ from $\ell_2$ to $E \subset \A$, so that $E$ is complemented by
a projection $P$. For a fixed $p \in (2,\infty)$, consider an isometry
$i_p : \ell_p \to C[0,1] \subset \A$, and the formal identity $i : \ell_2 \to \ell_p$.
Then $S = i_2 q$ and $T = i_p i i_2^{-1} P$ have the desired properties.
\end{proof}

\bibliographystyle{amsplain}

%    Insert the bibliography data here.

\end{document}